\documentclass{amsart}

\usepackage{lineno,hyperref}
\modulolinenumbers[5]

\usepackage{amssymb,amsmath}
\usepackage{amsthm}
\usepackage{a4wide}
\usepackage{xcolor}

 \usepackage[utf8]{inputenc}

\newcommand{\Com}{\ensuremath{C^m_0}}

\newcommand{\real}{\ensuremath{\mathbb{R}}}
\newcommand{\rd}{\ensuremath{\mathbb{R}^d}}
\newcommand{\nat}{\ensuremath{\mathbb{N}}}
\newcommand{\no}{\ensuremath{\nat_0}}
\newcommand{\comp}{\ensuremath{\mathbb{C}}}

\newcommand{\dint}{\ensuremath{\mathrm{d}}}
\newcommand{\supp}{\ensuremath{\operatorname{supp}}}
\newcommand{\id}{\ensuremath{\operatorname{id}}}
\newcommand{\rank}{\ensuremath{\operatorname{rank}}}

\newcommand{\spa}{\ensuremath{\operatorname{span}}}

\newcommand{\dom}{\ensuremath{\operatorname{dom}}}
\newcommand{\Dd}{\ensuremath{\mathrm{D}}}
\newcommand{\vol}{\ensuremath{\operatorname{vol}}}

\newcommand{\bli}{\begin{list}{}{\labelwidth1.7em\leftmargin2.1em}}
\newcommand{\eli}{\end{list}}
\newtheorem{lemma}{Lemma}[section]

\newtheorem{prop}[lemma]{Proposition}
\newtheorem{theo}[lemma]{Theorem}

\theoremstyle{definition}

\newtheorem{exams}[lemma]{Examples}

\theoremstyle{remark}
\newtheorem{rem}[lemma]{Remark}

\numberwithin{equation}{section}

\begin{document}


\title{Sharp estimates for approximation numbers of non-periodic Sobolev embeddings}

\author{Therese Mieth}
\address{Mathematisches Institut, Abteilung Funktionalanalysis, Fakult\"at f\"ur Mathematik und Informatik, Universit\"at Leipzig, Postfach 100920, 04009 Leipzig, Germany}
\email{therese.mieth@math.uni-leipzig.de}

\begin{abstract}
We investigate asymptotically sharp upper and lower bounds for the approximation numbers of the compact Sobolev embeddings $\overset{\circ}{W}{^m}(\Omega)\hookrightarrow L_2(\Omega)$ and $ W^m(\Omega)\hookrightarrow L_2(\Omega)$, defined on a bounded domain $\Omega\subset\mathbb{R}^d$, involving explicit constants depending on $m$ and $d$. The key of proof is to relate the approximation problems to certain Dirichlet and Neumann eigenvalue problems.
\end{abstract}




\maketitle

\section{Introduction}

We investigate the asymptotic behaviour of the approximation numbers of the embeddings
\[
W^m(\Omega)\hookrightarrow L_2(\Omega)\qquad 
\text{ and }\qquad
\overset{\circ}{W}{^m}(\Omega)\hookrightarrow L_2(\Omega),\quad m\in\nat,
\]
where $W^m(\Omega)=\left\{f\in L_2(\Omega): \exists \Dd^\alpha f \in L_2(\Omega)\, \forall \alpha\in\no^d, |\alpha|\leq m \right\}$ is the usual Sobolev space defined on a bounded domain $\Omega\subseteq\mathbb{R}^d$ and $\overset{\circ}{W}{^m}(\Omega)$ is the closure of $C_0^\infty(\Omega)$ in $W^m(\Omega)$. For a linear and bounded operator $T\in \mathcal{L}(X,Y)$ acting between two Banach spaces $X,Y$ the $k$th- approximation number is defined by
\begin{equation}
a_k(T):= \inf_{\substack{S\in\mathcal{L}(X,Y)\\\rank S<k}}\Vert T-S\Vert_{X\to Y}
=
\inf_{\substack{S\in\mathcal{L}(X,Y)\\\rank S<k}}\sup_{\substack{x\in X\\ x\neq 0}}\frac{\Vert Tx - Sx\Vert_Y}{\Vert x\Vert_X}
, \quad k\in\nat.
\end{equation} 
Approximation numbers are well-established concepts that serve in many situations as an effective measure for compactness especially in Hilbert space settings. The asymptotic behaviour of the sequence $(a_k)_{k\in\nat}$ for classical Sobolev embeddings is well known. In fact it is
\begin{equation}\label{class_estim1}
a_k(\id: W^m(\Omega)\hookrightarrow L_2(\Omega))\sim k^{-\frac{m}{d}}
\end{equation}
meaning that there exist constants $C,c>0$ such that for all $k\in\nat$ it holds
\begin{equation}\label{class_estim}
c\cdot k^{-\frac{m}{d}}\leq a_k(\id: W^m(\Omega)\hookrightarrow L_2(\Omega)) \leq C\cdot k^{-\frac{m}{d}}.
\end{equation}
Note that the latter also holds true for $\overset{\circ}{W}{^m}(\Omega)$ instead of $W^m(\Omega)$. This classical result goes back to \textsc{Kolmogorov} \cite{Kol36}. We refer to \cite[Theorem 4.10.2, p.348]{Tri78}, \cite[Sect.V.6, p.292]{EE87} and references given in subsequent remarks therein. However nowadays there is an increasing interest in the hidden dependence of the constants in \eqref{class_estim} on certain parameters like the dimension of the underling domain or the chosen norm. In fact although many authors dealt with the investigation of approximation numbers (and other concepts) of various embeddings of functions spaces, it was a frequent practice to involve only unspecific constants. There have been established some results including the $d$-dependence of these constants in the case of mixed order Sobolev functions on the $d$-torus by \textsc{D\~{u}ng/Ullrich} \cite{DuUl13} and in a similar framework by \textsc{Chernov/D\~{u}ng} \cite{ChDu16}. Super-exponential decay of such constants in $d$ has been observed already by \textsc{Bungartz/Griebel} \cite{BuGr04}, see also similar results due to \textsc{Griebel/Knapek} \cite{GrKn00, GrKn09} \textsc{Griebel} \cite{Gri06} and \textsc{Schwab/S\"uli/Todor} \cite{Sch.et.al08}. Further results with a particular emphasis on the hidden structure of optimal constants are determined by \textsc{Kühn/Sickel/Ullrich} \cite{KSU14} for periodic functions defined on the $d$-torus $\mathbb{T}^d$. Inter alia these authors proved that the following limit exists
\begin{equation}\label{KSU_lim}
\lim_{k\to\infty}\frac{a_k(H^{s,p}(\mathbb{T}^d)\hookrightarrow L_2(\mathbb{T}^d))}{k^{-\frac{s}{d}}} = (\vol B^d_p)^{\frac{s}{d}}
\end{equation}
where for $1\leq p<\infty$ the periodic Sobolev norm with fractional smoothness $s>0$ is given by $\Vert f| H^{s,p}(\mathbb{T}^d)\Vert = \left(\sum\limits_{\bar{k}\in\mathbb{Z}^d} \big( 1+ \sum\limits_{j=1}^d |k_j|^p \big)^{\frac{2s}{p}} |\hat{f}(\bar{k})|^2\right)^{1/2}$, a weighted $\ell_2$-sum of the Fourier coefficients $\hat{f}(\bar{k})$, and $B^d_p=\{x\in\rd: \sum\limits_{j=1}^d |x_j|^p \leq 1\}.$ Furthermore \cite{KSU14} contains estimates of the form \eqref{class_estim} with explicit constants for large $k\geq k_0$. For instance if $p=2s$ and $k\geq 11^d e^{\frac{d}{2s}}$ the named authors showed
\begin{equation}\label{KSU_estim}
\frac{1}{\sqrt{e(d+2s)}}\cdot k^{-\frac{s}{d}}
\leq 
a_k(\id: H^{s,2s}(\mathbb{T}^d)\hookrightarrow L_2(\mathbb{T}^d))
\leq
\frac{4^s}{\sqrt{d}}\sqrt{2em}\cdot k^{-\frac{s}{d}}.
\end{equation}
Further developments in the periodic setting of Sobolev spaces with fractional smoothness and with mixed order smoothness were made in \cite{CKS16, KMU16,KSU15}. In the non-periodic case even less is known. We refer to \textsc{Krieg} \cite{Krie18} for some results on the approximation of mixed order Sobolev functions on the cube.\\

\noindent
The aim of the present paper is to extend results like \eqref{KSU_lim} and \eqref{KSU_estim} to the non-periodic situation and focus on the constants that appear in the classical asymptotic result \eqref{class_estim}. Therefore we point out an important connection between the approximation numbers and certain eigenvalue problems. The key idea is to use the fact that the approximation numbers of a dense, compact embedding $X\hookrightarrow Y$ between two Hilbert spaces are related with the eigenvalues of a densely defined, positive definite, self-adjoint operator $A:\dom A\subseteq X\to Y$ with pure point spectrum which is uniquely determined by 
\[
\langle f,g\rangle_X=\langle A f, g\rangle_Y \qquad \forall\, f\in\dom A, g\in X.
\]
In that sense we say that the operator $A$ is norm-inducing in $X$. If $X$ is the Sobolev space  $\overset{\circ}{W}{^m}(\Omega)$  or $W^m(\Omega)$ and $Y=L_2(\Omega)$ then the operator $A$ is nothing else than a linear partial differential operator of order $2m$ with respect to Dirichlet or Neumann boundary conditions respectively. The mathematical background is described in Section 2. Using this approach we focus on the asymptotic distribution of corresponding eigenvalues and are able to give the asymptotic constants in the non-periodic setting, similar to \eqref{KSU_lim}. We derive at the beginning of Section 3 that the limit
\begin{equation}
\lim_{k\to\infty}\frac{ a_k(\id:\overset{\circ}{W}{^m}(\Omega)\hookrightarrow L_2(\Omega))}{k^{-\frac{m}{d}}}=\Big(\frac{\vol\Omega}{(2\pi)^d} \vol A^d_m\Big)^{\frac{m}{d}}
\end{equation}
exists where the set $A^d_m=\{z\in\rd: a(z)<1\}$ is determined by various equivalent homogeneous Sobolev norms. We get in Theorem \ref{theo_LiYau} sharp upper bounds
\[
 a_k(\id:\overset{\circ}{W}{^m}(\Omega)\hookrightarrow L_2(\Omega)) \leq \sqrt{\tfrac{2m+d}{d}}\cdot\Big(\frac{\vol \Omega}{(2\pi)^d}\vol A^d_m\Big)^{\frac{m}{d}}\cdot k^{-\frac{m}{d}}
\]
for all $k\in\nat$. As a counterpart we show in Theorem \ref{theo_kroe} lower bounds 
\[
a_{k+1}(\id: W^m(\Omega)\hookrightarrow L_2(\Omega)) 
\geq 
\left[1\,+\, \left(\tfrac{2m+d}{2m}\right)^{\frac{2m}{d}}\, \Big(\frac{\vol\Omega}{(2\pi)^d}\,\vol A^d_m \Big)^{-\frac{2m}{d}} \cdot k^{\frac{2m}{d}}\right]^{-\frac{1}{2}}
\]
for all $k\in\nat$. Finally in Section 5 we deal with the special case of the natural Sobolev norm $\Vert f| W^m(Q)\Vert ^* := \big(\Vert f|L_2(Q)\Vert^2\,+\,\sum\limits_{j=1}^d \Vert \partial^m_jf|L_2(Q)\Vert^2\big)^{\frac{1}{2}}$ where only the highest partial derivatives in each direction are taken and $Q=[0,L]^d$ is a cube. The special feature here is that the eigenfunctions of the corresponding norm-inducing operator obey a tensor product structure. Hence the corresponding eigenvalues of the $d$-dimensional problem can be written in terms of the eigenvalues of the one dimensional problem. This similarity to the periodic case enables us to use techniques and arguments from \cite{KSU14}. As a result we get explicit lower and upper bounds for the approximation numbers of the non-periodic $*$-normed Sobolev embedding for large $k\geq k_0$
\[
\Big(\frac{L}{4\pi}\Big)^m \big(\vol B^d_{2m} \big)^{\frac{m}{d}}\cdot k^{-\frac{m}{d}} \leq
a_k(\id: W^{m,*}(Q)\hookrightarrow L_2(Q)) \leq  \Big(\frac{L}{\pi}\Big)^m \big(\vol B^d_{2m} \big)^{\frac{m}{d}}\cdot k^{-\frac{m}{d}}
\]
where $k_0=k_0(m,d)$ is explicitly given in Theorem \ref{theo_norm*}.

\section{Preliminaries}\label{sec_prelim}

We use standard notation. $\nat$ is the set of all natural numbers, $\no = \nat\cup\{0\}$, $\real$ is the set of all real numbers, $\rd,d\in\nat,$ is the Euclidean $d-$space and $\rd_+=\{z=(z_1,...,z_d)\in\rd: z_i>0, i=1,...,d\}$. For $s\in\real$ we denote $s_+:=\max\,\{0,s\}.$ Let $\lceil t\rceil:=\min\{n\in\nat: t\leq n\}$ be the smallest integer not less than $t>0$, hence $\lceil t\rceil-1<t\leq \lceil t\rceil.$
For $0<p<\infty$ and $x\in\rd$ we denote $\Vert x\Vert_p:=\Big(\sum\limits_{j=1}^d |x_j|^p\Big)^{1/p}$. We write $\ell_p^d$ for $\rd$ equipped with the norm $\Vert\cdot|\ell_p^d\Vert:=\Vert\cdot\Vert_p$. By $B^d_p$ we denote the closed unit ball of $\ell_p^d$.

 $\Com(\Omega)$ collects all complex-valued functions $f$ on $\rd$ having classical derivatives up to order $m\in\no$ with compact support $\supp f$ in $\Omega\subset\rd$.  $L_2(\Omega)$ and all other spaces introduced below are considered in the standard setting of distributions and $\langle f, g\rangle_{L_2(\Omega)}=\int\limits_{\Omega} f(x)\overline{g(x)}\dint x.$ For $f\in L_1(\rd)$ we define the Fourier transform $\mathcal{F}f$ of $f$ by 
\[
(\mathcal{F}f)(z):=(2\pi)^{-\frac{d}{2}}\int_{\rd} f(x)e^{-i x\cdot z}\dint x,\ z\in\rd.
\]
For a finite set $M$ we denote by $\# M$ the number of elements in $M$. Denote by $\vol\Omega$ the Lebesgue measure of a domain $\Omega\subset\rd.$

\subsection{Operators with pure point spectrum}
We start with some basic facts about positive definite, self-adjoint operators in Hilbert spaces. For details we refer to Chapter 4 in \cite{Tri92}. Let  $\mathcal{A}:\dom\mathcal{A}\to H$ be a densely defined, linear operator acting on the Hilbert space $H$ with inner product $\langle\cdot,\cdot\rangle$ and norm $\Vert\cdot\Vert:=\sqrt{\langle\cdot,\cdot\rangle}$. $\mathcal{A}$ is said to be \textit{positive definite}, if
\begin{equation}\label{def_pd}
\exists c>0\, \forall x\in \dom\mathcal{A}: \langle \mathcal{A}x, x\rangle \geq c \Vert x\Vert^2.
\end{equation}
The energy space $H_{\mathcal{A}}$ of the operator $\mathcal{A}$ is defined as
\begin{equation}\label{def_energyspace}
H_{\mathcal{A}}:=\left\{ x\in H\,|\, \exists (x_n)_{n\in\nat}\subseteq \dom\mathcal{A}: \Vert x-x_n\Vert \overset{n\to\infty}\longrightarrow 0,\, \Vert x_m-x_n|H_{\mathcal{A}}\Vert \overset{n,m\to\infty}\longrightarrow 0 \right\}
\end{equation}
where $\Vert x|H_{\mathcal{A}}\Vert:= \sqrt{\langle \mathcal{A}x,x\rangle}$ for $x\in\dom\mathcal{A}$. $H_{\mathcal{A}}$ becomes a Hilbert space furnished with the inner product
\begin{equation}
\left[ x,y \right]_{H_{\mathcal{A}}}:= \lim_{n\to\infty} \langle \mathcal{A}x_n,y_n\rangle
\end{equation}
since the limit exists and is independent  of the choice of the approximation sequence in \eqref{def_energyspace}. Let $\mathcal{A}_F$ denote the Friedrichs extension of $\mathcal{A}$, i.e. $\dom\mathcal{A}_F:=H_{\mathcal{A}}\cap \dom \mathcal{A}^*, \mathcal{A}_F f:=\mathcal{A}^* f$, where $\mathcal{A}^* $ is the adjoint operator. $\mathcal{A}_F$ is a self-adjoint extension of $\mathcal{A}$ and 
\begin{equation}
\forall x\in \dom\mathcal{A}_F: \langle \mathcal{A}_Fx, x\rangle \geq c \Vert x\Vert^2
\end{equation}
with the same constant $c>0$ as in \eqref{def_pd}. As standard, the operator $(\mathcal{A}_F)^{\frac{1}{2}}$ is defined via spectral representation. Then
\begin{equation}
H_{\mathcal{A}}= H_{\mathcal{A}_F} = \dom (\mathcal{A}_F)^{\frac{1}{2}}
\end{equation}
and $(\mathcal{A}_F)^{\frac{1}{2}}$ provides a unitary mapping of $H_{\mathcal{A}}$ onto $H$
\begin{equation}
\forall x\in H_{\mathcal{A}}: \Vert x|H_{\mathcal{A}}\Vert = \Vert (\mathcal{A}_F)^{\frac{1}{2}}x\Vert.
\end{equation}
A positive definite, self-adjoint operator $A$ is called \textit{an operator with pure point spectrum} if its spectrum consists solely of eigenvalues of finite multiplicity. We list some properties of such operators:
\begin{itemize}
\item[(i)]
$A$ is unbounded i.e. the eigenvalues of $A$, monotonically ordered according to their multiplicity, satisfy
\[
0<\lambda_1\leq\lambda_2\leq ...\leq \lambda_k\overset{k\to\infty}{\longrightarrow}\infty.
\]
\item[(ii)]
The system of corresponding orthonormal eigenvectors $\{x_k:k\in\nat\}$ is complete.
\item[(iii)]
$\dom A =\left\{x\in H: \sum\limits_{k=1}^\infty \lambda_k^2|\langle x,x_k\rangle|^2<\infty\right\}$ and $Ax=\sum\limits_{k=1}^\infty \lambda_k\langle x,x_k\rangle x_k$ with convergence in $H$.
\item[(iv)]
$H_A=\left\{ x\in H:\sum\limits_{k=1}^\infty \lambda_k|\langle x,x_k\rangle|^2<\infty\right\}$ and $\Vert x| H_A\Vert^2=\sum\limits_{k=1}^\infty \lambda_k|\langle x,x_k\rangle|^2$.
\item[(v)]
$A^{-1}$ and $A^{-\frac{1}{2}}=(A^{\frac{1}{2}})^{-1}$ are compact operators from $H$ to $H.$ $A^{-1}$ regarded as a mapping from $H$ into $H_A$ is compact, too.
\end{itemize}

The following criterion is due to \textsc{Rellich}.
\begin{prop}\label{prop_rell}
Let $A$ be a densely defined, positive definite and self-adjoint operator.  Then $A$ is an operator with pure point spectrum if and only if the embedding $\id: H_A\hookrightarrow H$ is compact. In that case it holds
\begin{equation}\label{rellich}
a_k(\id: H_A\hookrightarrow H) = \lambda_k^{-\frac{1}{2}},\quad k\in\nat,
\end{equation}
where $\lambda_k$ denotes the $k^{\textup{th}}$ eigenvalue of $A$.
\end{prop}
\begin{proof}
The if-and-only-if statement is well known, cf.\cite{Tri92}[Section 4.5.3, p.258]. We give the proof of \ref{rellich}. Let $\{x_k:k\in\nat\}$ be the set of corresponding orthonormal eigenvectors. Let $T:H\to H$ be a linear and bounded operator with $\rank T<k.$ There exist $\mu_1,...,\mu_k\in\comp$ such that $\sum\limits_{i=1}^k|\mu_i|^2=1\ \text{ and } Ty = 0\text{ where } y=\sum\limits_{i=1}^k \mu_i x_i.$ Hence
\[
\Vert A^{-\frac{1}{2}} - T\Vert_{H\to H}^2 \geq \Vert A^{-\frac{1}{2}} y\Vert^2=\sum_{i=1}^k\lambda_i^{-1}|\mu_i|^2\geq \lambda_k^{-1}.
\]
On the other hand we consider the operator $T_kx:=\sum\limits_{i=1}^{k-1} \lambda_i^{-\frac{1}{2}}\langle x,x_i\rangle x_i$. Then
\[
a_k(A^{-\frac{1}{2}}:H\to H)\leq \Vert T_k - A^{-\frac{1}{2}}\Vert_{H\to H} \leq \lambda_k^{-\frac{1}{2}}.
\]
Finally $a_k(A^{-\frac{1}{2}}:H\to H)=a_k(\id: H_A\hookrightarrow H)$ follows by the unitarity of $A^{\frac{1}{2}}:H_A\to H$ (and hence of $A^{-\frac{1}{2}}:H\to H_A$) and the multiplicativity of approximation numbers.
\end{proof}

\subsection{Characterization of Sobolev spaces by spectral properties of norm-inducing operators}

Let $\Omega\subset\rd$ be a bounded domain and let 
\[
W^m(\Omega)=\left\{f\in L_2(\Omega): \exists \Dd^\alpha f \in L_2(\Omega)\, \forall \alpha\in\no^d, |\alpha|\leq m \right\}
\] 
be the classical Sobolev space. Let us assume that the norm in $W^m(\Omega)$ can be written as
\begin{equation}
\Vert f| W^{m}(\Omega)\Vert_{(b_\alpha)}:= \Big(\Vert f| L_2(\Omega)\Vert^2+\sum_{1\leq|\alpha|\leq m}b_\alpha \Vert\Dd^\alpha f|L_2(\Omega)\Vert^2  \Big)^{\frac{1}{2}}
\end{equation}
with coefficients $b_\alpha\in\{0,1\}$. Similarly assume that the norm in $\overset{\circ}{W}{^m}(\Omega)$, the closure of $C_0^\infty(\Omega)$ in $W^m(\Omega)$, can be written as
\[
\Vert f| \overset{\circ}{W}{^m}(\Omega)\Vert_{(a_\alpha)}:= \Big(\sum_{|\alpha|\leq m}a_\alpha \Vert\Dd^\alpha f|L_2(\Omega)\Vert^2  \Big)^{\frac{1}{2}}
\]
with coefficients $a_\alpha\in\{0,1\}.$ Here we only consider $a_\alpha$'s and $b_\alpha$'s such that the corresponding norms are equivalent to the standard Soblev norm where $a_\alpha=b_\alpha=1$ for all $\alpha\in\no^d, |\alpha|\leq m.$
 
\begin{prop}
Let $\Omega\subseteq\rd$ be a bounded domain. The operator 
\[
(\mathcal{A}f)(x)=\sum\limits_{|\alpha|\leq m}{(-1)^{|\alpha|}}a_\alpha \Dd^{2\alpha} f(x)
\] 
defined on $\dom\mathcal{A}= C_0^\infty(\Omega)$ is positive definite and symmetric in $L_2(\Omega)$. Its Friedrichs extension $A$ has pure point spectrum consisting of eigenvalues of finite multiplicity
\[
0<\lambda_1\leq\lambda_2\leq ...\leq \lambda_k\overset{k\to\infty}{\longrightarrow}\infty.
\]
If $\{\psi_k:k\in\nat\}$ denotes the corresponding sequence of orthonormal eigenfunctions of $A$, it holds
\begin{equation}
\forall f\in\overset{\circ}{W}{^m}(\Omega):\ 
\Vert f| \overset{\circ}{W}{^m}(\Omega)\Vert_{(a_\alpha)}^2= \sum_{k=1}^\infty \lambda_k\,|\langle f,\psi_k\rangle_{L_2(\Omega)}|^2.
\end{equation}
Furthermore for all $k\in\nat$ we have
\begin{equation}\label{ak_W0^m}
a_k(\id:\overset{\circ}{W}{^m}(\Omega)\hookrightarrow L_2(\Omega))=\lambda_k^{-\frac{1}{2}}.
\end{equation}
\end{prop}

\begin{proof}
Due to our previous observations, in particular Proposition \ref{prop_rell}, we only have to be sure that the energy space of $\mathcal{A}$ coincides with $\overset{\circ}{W}{^m}(\Omega)$.  Let $f\in H_{\mathcal{A}}$, i.e. there exists $(f_n)_n\subseteq C^\infty_0(\Omega)$ such that $\Vert f-f_n|L_2(\Omega)\Vert\overset{n\to\infty}{\longrightarrow}0$ and $\Vert f_m-f_n|\overset{\circ}{W}{^m}(\Omega)\Vert_{(a_\alpha)}\overset{n,m\to\infty}{\longrightarrow}0$. Due to completeness there is a $g\in\overset{\circ}{W}{^m}(\Omega)$ with $\Vert g-f_n|\overset{\circ}{W}{^m}(\Omega)\Vert_{(a_\alpha)}\overset{n\to\infty}{\longrightarrow}0$. Hence
\[
\Vert f-g|L_2(\Omega)\Vert \leq \Vert f-f_n|L_2(\Omega)\Vert + c\,\Vert g-f_n|\overset{\circ}{W}{^m}(\Omega)\Vert_{(a_\alpha)} \overset{n\to\infty}{\longrightarrow}0,
\]
so $f\equiv g$. On the other hand the inclusion $\overset{\circ}{W}{^m}(\Omega)\subseteq H_{\mathcal{A}}$ is also clear and we have for all $f\in C_0^\infty(\Omega)=\dom \mathcal{A}$
\begin{equation}\label{norm=1}
\Vert f|\overset{\circ}{W}{^m}(\Omega)\Vert_{(a_\alpha)}^2 = \langle \mathcal{A} f, f\rangle = \Vert f| H_A\Vert^2.
\end{equation}
\end{proof}

\begin{prop}\label{prop_sobolevW^m}
Let $\Omega\subseteq\rd$ be a bounded Lipschitz domain and $\nu$ denote the outward unit normal to $\Omega$. The operator 
\[
(\mathcal{B}f)(x)=\sum\limits_{1\leq|\alpha|\leq m}{(-1)^{|\alpha|}}b_\alpha \Dd^{2\alpha} f(x)
\]
defined on $\dom\mathcal{B}= C^{m,\nu}(\Omega):=\{f\in C^\infty(\overline{\Omega}):\frac{\partial^j f}{\partial \nu^j}=0 \text{ a.e. on } \partial\Omega, j=m,...,2m-1\}$ is non-negative and symmetric in $L_2(\Omega)$. Its Friedrichs extension $B$ has pure point spectrum consisting of eigenvalues of finite multiplicity
\[
0=\mu_1\leq\mu_2\leq ...\leq \mu_k\overset{k\to\infty}{\longrightarrow}\infty.
\]
If $\{\phi_k:k\in\nat\}$ denotes the corresponding sequence of orthonormal eigenfunctions of $B$, it holds
\begin{equation}
\forall f\in W{^m}(\Omega):\ 
\Vert f| W{^m}(\Omega)\Vert_{(b_\alpha)}^2= \sum_{k=1}^\infty (1+\mu_k)\,|\langle f,\phi_k\rangle_{L_2(\Omega)}|^2.
\end{equation}
Furthermore for all $k\in\nat$ we have
\begin{equation}\label{ak}
a_k(\id:W{^m}(\Omega)\hookrightarrow L_2(\Omega))=(1+\mu_k)^{-\frac{1}{2}}.
\end{equation}
\end{prop}
\begin{proof}
Recall that $C^{m,\nu}(\Omega)$ is dense in $W^m(\Omega)$ and that for $f\in C^{m,\nu}(\Omega$ we have
\begin{equation}\label{norm=2}
\Vert f| W{^m}(\Omega)\Vert_{(b_\alpha)}^2= \Vert f|L_2(\Omega)\Vert^2 + \langle Bf,f\rangle = \Vert f|L_2(\Omega)\Vert^2 + \Vert f| H_B\Vert^2.
\end{equation}
\end{proof}
We say that the positive definite, self-adjoint operator $A$ (resp. $\id +B$) with pure point spectrum  is \textit{norm-inducing} in the Hilbert space $\overset{\circ}{W}{^m}(\Omega)$  (resp. $W^m(\Omega)$). In the next proposition we recall the Rayleigh-Ritz variational characterisations of eigenvalues.

\begin{prop}
Let $\Omega\subseteq\rd$ be a bounded Lipschitz domain. As introduced in Proposition \ref{prop_sobolevW^m}, let $\{\phi_k:k\in\nat\}$ be the complete set of orthonormal eigenfunctions of the operator $B$ corresponding to the eigenvalues $\{\mu_k: k\in\nat\}$.  Then it holds
\begin{equation}\label{varchar}
1+\mu_k
=
\sup_{f\in\spa\{\phi_1,...,\phi_k\}}\varrho(f)
=
\inf_{\substack{f\in W^m(\Omega)\\ f\perp\phi_j, j=1,...,k-1}} \varrho(f)
=
\inf_{\substack{U\subset W^m(\Omega)\\\dim U = k}} \sup_{f\in U}\varrho(f)
\end{equation}
where the Rayleigh-quotient is defined by
\[
\varrho(f):=\frac{\Vert f\,| W^m(\Omega)\Vert^2_{(b_\alpha)}}{\Vert f\,|L_2(\Omega)\Vert^2}.
\]
\end{prop}

\begin{proof}
\textbf{Step 1.} It holds
\[
1+ \mu_k 
= \sum_{j=1}^\infty (1+\mu_j)|\langle \phi_k,\phi_j\rangle_{L_2(\Omega)}|^2 
=
\Vert \phi_k\,| W^m(\Omega)\Vert^2_{(b_\alpha)}
= 
\varrho(\phi_k)
\leq
\sup_{f\in\spa\{\phi_1,...,\phi_k\}}\varrho(f). 
\]
On the other hand for every $f\in\spa\{\phi_1,...,\phi_k\}$ we have $\langle f,\phi_j\rangle_{L_2(\Omega)}=0$ if $j\geq k+1$ and hence
\[
\Vert f\,|W^m(\Omega)\Vert^2_{(b_\alpha)} = \sum_{j=1}^k (1+\mu_j)|\langle f,\phi_j\rangle_{L_2(\Omega)}|^2 \leq (1+\mu_k)\Vert f\,|L_2(\Omega)\Vert^2.
\] 
The second equality in \eqref{varchar} follows similarly. 
\textbf{Step 2.} Let $U\subset W^m(\Omega)$ be a linear subspace with $\dim U =k$. Then there exists an element $0\neq g\in U$ such that $g\perp \phi_j,j=1,...,k-1.$ It holds
\[
\varrho(g)\geq\inf_{\substack{f\in W^m(\Omega)\\ f\perp\phi_j, j=1,...,k-1}} \varrho(f) = 1+\mu_k.
\]
This shows
\[
\inf_{\substack{U\subset W^m(\Omega)\\\dim U = k}} \sup_{f\in U}\varrho(f)
\geq 1+\mu_k.
\]
The inverse inequality follows with the first equality of \eqref{varchar} by taking $U:=\spa\{\varphi_1,...,\varphi_k\}.$
\end{proof}

Clearly corresponding variational characterisations also hold for the Sobolev space $\overset{\circ}{W}{^m}(\Omega)$. With the introduced notation this is
\begin{equation}\label{varchar_D}
\lambda_k
=
\sup_{f\in\spa\{\psi_1,...,\psi_k\}}\tilde{\varrho}(f)
=
\inf_{\substack{f\in \overset{\circ}{W}{^m}(\Omega)\\ f\perp\psi_j, j=1,...,k-1}} \tilde{\varrho}(f)
=
\inf_{\substack{U\subset \overset{\circ}{W}{^m}(\Omega)\\\dim U = k}} \sup_{f\in U}\tilde{\varrho}(f)
\end{equation}
where
\[
\tilde{\varrho}(f):=\frac{\Vert f\,| \overset{\circ}{W}{^m}(\Omega)\Vert^2_{(a_\alpha)}}{\Vert f\,|L_2(\Omega)\Vert^2}.
\]

\section{Sharp upper bounds for approximation numbers of $\id:\overset{\circ}{W}{^m}(\Omega)\hookrightarrow L_2(\Omega)$}

As described before let the Sobolev space $\overset{\circ}{W}{^m}(\Omega)$ be endowed with the inner product
\begin{equation}
\langle f, g \rangle_{\overset{\circ}{W}{^m}(\Omega),(a_\alpha)}=\sum_{|\alpha|\leq m}a_\alpha\,\langle \Dd^\alpha f, \Dd^\alpha g\rangle_{L_2(\Omega)}
\end{equation}
with coefficients $a_\alpha\in\{0,1\}$ and the norm-inducing operator $A$ be given by
\begin{equation}\label{def_opA}
Af = \sum\limits_{|\alpha|\leq m}{(-1)^{|\alpha|}}a_\alpha \Dd^{2\alpha} f
\end{equation}
satisfying Dirichlet boundary conditions
\begin{equation}\label{def_Dbc}
\frac{\partial^j f}{\partial\nu^j}=0 \text{ on } \partial\Omega, j=0,...,m-1
\end{equation}
where \eqref{def_opA} and \eqref{def_Dbc} are first defined on $C^\infty_0(\Omega)$ and then uniquely extended to $\overset{\circ}{W}{^m}(\Omega)$. We have
\[
\langle f, g \rangle_{\overset{\circ}{W}{^m}(\Omega),(a_\alpha)} = \langle A f, g\rangle_{L_2(\Omega),}\quad f\in\dom A, g\in \overset{\circ}{W}{^m}(\Omega). 
\]
Define a corresponding polynomial
 of the operator $A$ by
\[
a(z):= \sum_{|\alpha|\leq m}a_\alpha z^{2\alpha},\quad z\in\rd,
\]
where $z^{2\alpha}:=z_1^{2\alpha_1}\cdot...\cdot z_d^{2\alpha_d}.$
In view of \eqref{ak_W0^m} we need to analyse the asymptotic behaviour of the eigenvalues of $A$. However many papers are concerned with the investigation on distributions of eigenvalues for elliptic operators. The underlying theory goes back to \textsc{Weyl, Courant} and \textsc{Carleman}. A treatment of the history can be found in \cite{Cla67}. Further extensive developments are due to \textsc{G\r{a}rding, Browder} in the fifties and \textsc{Agmon} a bit later, cf. \cite{Agm65}. The operator $A$ is a differential operator with constant coefficients. Denote by $N(\lambda)$ the number of eigenvalues of $A$ that are less than $\lambda>0$ and 
\[
V(\lambda):=\vol\{z\in\rd: a(z)<\lambda\}.
\]
Then it holds
\begin{equation}\label{tul71}
\lim_{\lambda\to\infty}\frac{N(\lambda)}{V(\lambda)}= \frac{\vol \Omega}{(2\pi)^d}.
\end{equation}
A direct proof of \eqref{tul71} can be found in \cite{Tul71}. We remind of the following lemma.
\begin{lemma}\label{Lemma_V}
Let  $a:\rd\to\real$ be a non-negative function that is homogeneous of order $2m$, i.e. $a(\lambda z)=\lambda^{2m}a(z)\, ,z\in\rd, \lambda\in\real$ and $a(z)=a(|z_1|,...,|z_d|)$ for all $z=(z_1,...,z_d)\in\rd$. Then we have for any $\lambda>0$ 
\begin{equation}
V(\lambda)= V(1)\,\lambda^{\frac{d}{2m}}
\quad\text{and}\quad
\int_{\{z\in\rd: a(z)<\lambda\}} a(z)\dint z =\frac{d}{2m+d}\, V(1)\,\lambda^{\frac{2m+d}{2m}}.
\end{equation}
\end{lemma}
\begin{proof}
Apply the transformation formula for integrals using polar coordinates $\varphi:(0,\infty)\times S_+\to\rd_+$ defined by
\[
\varphi (r,y)= r\, y
\]
where $S_+=\{y\in\rd_+:a(y)=1 \}.$
\end{proof}
Thus for homogeneous norms of order $2m$ we can rewrite \eqref{tul71} as
\begin{equation}\label{asymptbehaviour}
\lim_{k\to\infty}
\frac{ a_k(\id:\overset{\circ}{W}{^m}(\Omega)\hookrightarrow L_2(\Omega))}{k^{-\frac{m}{d}}}
=
\Big(\frac{\vol\Omega}{(2\pi)^d} \vol A^d_m\Big)^{\frac{m}{d}}
\end{equation}
where 
\begin{equation}\label{A^d_m}
A^d_m:=\{z\in\rd: a(z)<1\}.
\end{equation}
If $m=1$ and $a(z)=|z|^2$ this coincides with the famous Weyl formula \cite{Wey12} for eigenvalues of the Dirichlet-Laplace operator.
The next theorem gives sharp estimates with optimal constants.
\begin{theo}\label{theo_LiYau}
Let $\Omega\subset\rd$ be a bounded domain and $a_k(\id)$ denote the $k^{\textup{th}}$ approximation number of the compact embedding
$
\id:\overset{\circ}{W}{^m}(\Omega)\hookrightarrow L_2(\Omega).$
Assume that the Sobolev norm is of the form
\[\Vert f| \overset{\circ}{W}{^m}(\Omega)\Vert_{(a_\alpha)}:= \Big(\sum\limits_{|\alpha|= m}a_\alpha \Vert\Dd^\alpha f|L_2(\Omega)\Vert^2  \Big)^{\frac{1}{2}},\qquad a_\alpha\in\{0,1\},\]
and consider the homogeneous polynomial $a(z)=\sum\limits_{|\alpha|= m}a_\alpha z^{2\alpha}$ of order $2m.$ Then for all $k\in\nat$
\begin{equation}
a_k(\id) \leq \sqrt{\frac{2m+d}{d}}\cdot\Big(\frac{\vol\Omega}{(2\pi)^d} \vol A^d_m\Big)^{\frac{m}{d}}\cdot k^{-\frac{m}{d}}.
\end{equation}
\end{theo}
In view of \eqref{asymptbehaviour} this upper bound for the approximation numbers is asymptotically sharp for $d\to\infty$. The proof of Theorem \ref{theo_LiYau} is inspired by \textsc{Li/Yau} \cite{LiYau83} what can be seen as an equivalent to our situation if $m=1$ and $a(z)=|z|^2.$ We start with the following lemma. 
\begin{lemma}\label{lemma1}
Let $a,f:\rd\to\real$ be non-negative functions such that there exist constants $M_1,M_2>0$ with
\[
0\leq f\leq M_1 \quad \text{ and }\quad \int_{\rd} a(z)f(z)\dint z \leq M_2.
\]
Then for every $R>0$ such that
\[
\int_{\{z\in\rd: a(z)<R\}}a(z)\dint z = \frac{M_2}{M_1}
\]
it holds
\[
\int_{\rd}f(z)\dint z \leq M_1 \vol\big\{z\in\rd: a(z)<R\big\}.
\]
\end{lemma}
\begin{proof}
For $z\in\rd$ define
\[
g(z):=\begin{cases} M_1\quad& \text{if }a(z)<R,\\
0\quad& \text{if }a(z)\geq R.\end{cases}
\]
Then $(a(z)-R)(f(z)-g(z))\geq 0.$ Integrating gives
\[
R\int_{\rd}(f(z)-g(z))\dint z
\leq 
\int_{\rd}a(z)(f(z)-g(z))\dint z 
\leq
M_2 - \int_{\rd}a(z)g(z)\dint z =0.
\]
This shows
\[
\int_{\rd}f(z)\dint z
\leq
\int_{\rd}g(z)\dint z = \int_{\{z\in\rd: a(z)<R\}}M_1 \dint z.
\]
\end{proof}
\noindent
Recall that for every $f\in C^\infty_0(\Omega)$ and $z=(z_1,...,z_d)\in\rd$ it holds
\[
z^{\alpha}(\mathcal{F}f)(z) 
=
(-i)^{|\alpha|}\big[\mathcal{F}(\Dd^{\alpha} f)\big](z), \qquad \alpha\in\no^d.
\]
Using Plancherel's Theorem this shows
\begin{equation}\label{planch}
\langle f, g \rangle_{\overset{\circ}{W}{^m}(\Omega),(a_\alpha)} = \int_{\rd} a(z)\mathcal{F}f(z) \overline{\mathcal{F}g(z)}\dint z\qquad \forall f,g\in C^\infty_0(\Omega).
\end{equation}
The last identity is crucial to prove Theorem \ref{theo_LiYau}.

\begin{proof}[Proof of Theorem \ref{theo_LiYau}]
Let $\{\psi_j: j=1,...,k\}$ be the set of orthonormal eigenfunctions for the eigenvalues $\{\lambda_j: j=1,...,k\}$ of the norm-inducing operator $A$. We consider the function defined by
\[
\Phi(x,y):= \sum_{j=1}^k \psi_j(x)\psi_j(y),\quad x,y\in\Omega.
\]
The Fourier transform of $\Phi$ in the $x$-variable is given by
\[
\hat{\Phi}(z,y)
=
\sum_{j=1}^k (\mathcal{F}\psi_j)(z)\psi_j(y)
=
\sum_{j=1}^k (2\pi)^{-\frac{d}{2}}\int_\Omega \psi_j(x)e^{-ix\cdot z}\dint x\, \psi_j(y).
\]
By definition this is nothing more than a multiple by $(2\pi)^{-\frac{d}{2}}$ of the projection $Ph_z$ of the function $h_z(x)=e^{-ix\cdot z}$ onto the linear subspace spanned by the first $k^\textup{th}$ eigenfunctions $\psi_1,...,\psi_k$, i.e.
\[
\hat{\Phi}(z,y)= (2\pi)^{-\frac{d}{2}} \sum_{j=1}^k \langle h_z, \psi_j\rangle_{L_2(\Omega)} \psi_j(y)=(2\pi)^{-\frac{d}{2}} (Ph_z)(y).
\]
We apply Lemma \ref{lemma1} to the function
\[
f(z):=\int_{\rd}|\hat{\Phi}(z,y)|^2\dint y,\quad z\in\rd
\]
which is the $L_2$-norm of $(2\pi)^{-\frac{d}{2}} Ph_z.$
It follows with Bessel's inequality 
\[
0\leq f(z) = (2\pi)^{-d}\sum_{j=1}^k |\langle h_z,\psi_j\rangle_{L_2(\Omega)} |^2 \leq (2\pi)^{-d} \Vert h_z| L_2(\Omega)\Vert^2 = \frac{\vol\Omega}{(2\pi)^d} =:M_1.
\]
Furthermore
\[
\int_{\rd}a(z)f(z)\dint z 
= 
\sum_{j=1}^k \int_{\rd} a(z)|\mathcal{F}\psi_j(z)|^2\dint z
= 
\sum_{j=1}^k \langle A\psi_j,\psi_j\rangle_{L_2(\Omega)} 
= 
\sum_{j=1}^k \lambda_j=: M_2
\]
where we used \eqref{planch}. Now if $R>0$ is chosen such that
\[
\tfrac{d}{2m+d}\, R^{\frac{2m+d}{2m}} \vol A_m^d = \frac{M_2}{M_1}
\ \iff\
R^{\frac{d}{2m}}= \Big(\sum_{j=1}^k \lambda_j\Big)^{\frac{d}{2m+d}}\Big( \tfrac{d}{2m+d}\, M_1 \vol A^d_m \Big)^{-\frac{d}{2m+d}}
\]
we conclude with Lemma \ref{lemma1} that
\[
k 
\leq 
M_1  R^{\frac{d}{2m}}\vol A_m^d 
= 
\big(M_1 \vol A^d_m\big)^{\frac{2m}{2m+d}}\, \Big(\sum_{j=1}^k \lambda_j\Big)^{\frac{d}{2m+d}} \big( \tfrac{d}{2m+d}\big)^{-\frac{d}{2m+d}},
\]
where we used
\[
\int_{\rd}f(z)\dint z = \sum_{j=1}^k\int_{\rd}|\mathcal{F}\psi_j(z)|^2\dint z = k.
\]
Hence
\[
\sum_{j=1}^k\lambda_j 
\geq
\tfrac{d}{2m+d}\, (M_1 \vol A^d_m)^{-\frac{2m}{d}}\, k^{\frac{2m+d}{d}}
\]
and obviously
\[
a_k(\id)^{-2}=\lambda_k\geq \frac{1}{k} \sum_{j=1}^k\lambda_j  \geq \tfrac{d}{2m+d}\, (M_1  \vol A^d_m)^{-\frac{2m}{d}}  k^{\frac{2m}{d}}.
\]
\end{proof}

\begin{exams}{Let $Q=(0,2\pi)^d$ be the $d$-dimensional cube with side length $2\pi$.}
\begin{itemize}
\item[1.]
The Sobolev norm 
\[
\Vert f| \overset{\circ}{W}{^m}(Q)\Vert ^* := \Big(\sum\limits_{j=1}^d \Vert \partial^m_jf|L_2(Q)\Vert^2\Big)^{1/2}
\]
is induced by the operator
$(-1)^m\sum\limits_{j=1}^d \partial^{2m}_j$ with corresponding Dirichlet boundary conditions. The corresponding polynomial is $a(z)= \sum\limits_{j=1}^d z_j^{2m}=\Vert z|\ell_{2m}^d\Vert^{2m}.$ Hence the Lebesgue measure of the set $A^{d,*}_m=\left\{z\in\rd: \Vert z|\ell_{2m}^d\Vert<1 \right\}=B^d_{2m}$ can be estimated as in \cite[Lemma 4.10]{KSU14} by
\[
(\vol A^{d,*}_m)^{\frac{m}{d}} \leq 2^m \Big(\frac{2em}{d}\Big)^{\frac{1}{2}}
\]
Therefore we have
\begin{equation}\label{up1}
a_k(\id:\overset{\circ}{W}{^{m,*}}(Q)\hookrightarrow L_2(Q))
\leq
\sqrt{\frac{2m+d}{d}} 2^m \Big(\frac{2em}{d}\Big)^{\frac{1}{2}} \cdot k^{-\frac{m}{d}},\quad k\in\nat.
\end{equation}
This upper bound slightly improves the upper bound in the periodic case given in \cite{KSU14}[Theorem 4.12], namely by the factor $\sqrt{\frac{2m+d}{d}}\,2^m$ instead of $4^m$. Furthermore \eqref{up1} holds for all $k\in\nat$ and not only for large $k\geq k_0(d,m).$
\item[2.]
Consider the Sobolev norm
\[
\Vert f| \overset{\circ}{W}{^m}(Q)\Vert ^+
:=
\Big(\sum\limits_{j_1=1}^d...\sum\limits_{j_m=1}^d \Vert \partial_{j_1}...\partial_{j_m} f|L_2(\Omega)\Vert^2\Big)^{\frac{1}{2}}
=
\Big(\int_{\rd}|z|^{2m}|\mathcal{F}f(z)|^2\dint z\Big)^{\frac{1}{2}}.
\]
The norm-inducing operator is $(-\Delta)^m= (-1)^m\sum\limits_{j_1=1}^d...\sum\limits_{j_m=1}^d \partial^2_{j_1}...\partial^2_{j_m}$ and $a(z)=|z|^{2m}$. Correspondingly the set $A^{d,+}_m = \vol\{z\in\rd: |z|<1\}$ is the unit ball in $\rd$. In particular 
\[
(\vol A^{d,+}_m)^{\frac{m}{d}} = \Big(\frac{\pi^{\frac{d}{2}}}{\Gamma(\frac{d}{2}+1)}\Big)^{\frac{m}{d}} 
\leq 
\Big(\frac{2\pi e}{d} \Big)^{\frac{m}{2}}
\]
where we used $\Gamma(1+x)\geq \left(\frac{x}{e}\right)^x,x\geq 0.$ Therefore we conclude
\begin{equation}
a_k(\id:\overset{\circ}{W}{^{m,+}}(Q)\hookrightarrow L_2(Q))
\leq
\sqrt{\frac{2m+d}{d}}\pi^{m/2}\cdot \Big(\frac{2 e}{d} \Big)^{m/2}\cdot k^{-m/d},\quad k\in\nat.
\end{equation}
Again this upper estimate replaces the factor $4^m$ from the periodic case \cite[Theorem 4.15]{KSU14} by $\sqrt{\frac{2m+d}{d}}\pi^{\frac{m}{2}}$ and holds for all $k\in\nat$.
\item[3.]
The introduced approach applies to fractional Sobolev space $\overset{\circ}{H}{^{s,p}}(\Omega), s>0,$ defined as the completion of $C^\infty_0(\Omega)$ with respect to the norm
\[
\Vert f| \overset{\circ}{H}{^{s,p}}(\Omega)\Vert := \Big(\int_{\rd}a_{s,p}(z)|\mathcal{F}f(z)|^2\dint z\Big)^{\frac{1}{2}}
\]
where $a_{s,p}(z)=\Big(\sum\limits_{j=1}^d |z_j|^p\Big)^{\frac{2s}{p}}, 0<p<\infty$. As a result we get for all $k\in\nat$
\[
a_k(\id:\overset{\circ}{H}{^{s,p}}(\Omega)\hookrightarrow L_2(\Omega) ) 
\leq 
\sqrt{\frac{2s+d}{d}}\cdot \left(\frac{\vol \Omega}{(2\pi)^d}\right)^{\frac{s}{d}} \cdot 
\begin{cases}
2^s\Big(\frac{ep}{d} \Big)^{\frac{s}{p}} \cdot k^{-\frac{s}{d}}\quad &\text{if } p\geq 1\\
2^s\Big(\frac{e(p+1)}{d} \Big)^{\frac{s}{p}} \cdot k^{-\frac{s}{d}}\quad &\text{if } 0<p<1.\\
\end{cases}
\]
\end{itemize}
\end{exams}

\section{Sharp lower bounds for approximation numbers of $\id:W^m(\Omega)\hookrightarrow L_2(\Omega)$}

\begin{theo}\label{theo_kroe}
Let $\Omega\subset\rd$ be a bounded Lipschitz domain. Suppose that $a_k(\id)$ denotes the $k^{\textup{th}}$ approximation number of the compact embedding $\id:W^m(\Omega)\hookrightarrow L_2(\Omega)$. Assume that the Sobolev norm is of the form
\[
\Vert f| W^m(\Omega)\Vert_{(b_\alpha)}:=  \bigg(\Vert f| L_2(\Omega)\Vert^2+\sum\limits_{|\alpha|= m}b_\alpha \Vert\Dd^\alpha f|L_2(\Omega)\Vert^2  \bigg)^{\frac{1}{2}}, \qquad b_\alpha\in\{0,1\},
\] 
and consider the homogeneous polynomial $a(z)=\sum\limits_{|\alpha|= m}b_\alpha z^{2\alpha}$ of order $2m.$
Then for all $k\in\nat$
\begin{equation}\label{Kroe_lowBd}
a_{k+1}(\id) 
\geq 
\left[1\,+\, \Big(\frac{2m}{2m+d}\Big)^{-\frac{2m}{d}}\, \Big(\frac{\vol\Omega}{(2\pi)^d}\,\vol A^d_m \Big)^{-\frac{2m}{d}} \cdot k^{\frac{2m}{d}}\right]^{-\frac{1}{2}}
\end{equation}
where $A^d_m=\{z\in\rd: a(z)<1\}.$
\end{theo}

Note that the expression $\Big(\frac{2m}{2m+d}\Big)^{m/d}$ tends to $1$ as $d\to\infty$.

\begin{proof}
The idea of the proof goes back to \textsc{Kröger} \cite{Kroe92} where the author considered the Neumann-Laplacian operator what can be seen as an equivalent to our situation in the case $m=1$ and $a(z)=|z|^2.$ We also refer to \cite{Lap97}. Let $\{\phi_j:j=1,...,k\}$ be the set of orthonormal eigenfunctions for the eigenvalues $\{\mu_j: j=1,...,k\}$ of the norm-inducing operator $B$ according to Proposition \ref{prop_sobolevW^m}. For $z\in\rd$ let the function
\[
h_z(x):=\begin{cases} 
e^{-ix\cdot z},\quad & x\in \Omega\\
0,\quad & x\notin \Omega.
\end{cases}
\]
We consider the function defined by
\[
\Phi(x,y):= \sum_{j=1}^k \phi_j(x)\phi_j(y),\quad x,y\in\Omega.
\]
The Fourier transform of $\Phi$ in the $x$-variable can be given in terms of the orthogonal projection $Ph_z$ of $h_z$ onto the subspace $\spa\{\phi_1,...,\phi_k\}$. Indeed
\[
\hat{\Phi}(z,y)
=
\sum_{j=1}^k (\mathcal{F}\phi_j)(z)\phi_j(y)
=
\sum_{j=1}^k (2\pi)^{-\frac{d}{2}}\langle h_z, \phi_j\rangle_{L_2(\Omega)} \phi_j(y)
=
(2\pi)^{-\frac{d}{2}} (Ph_z)(y).
\]
Due to the Rayleigh-Ritz variational formula \eqref{varchar} an upper bound for $\mu_{k+1}$ is given by 
\[
1+\mu_{k+1}
\leq
\frac{\int\limits_{B_r} \Vert h_z-Ph_z |W^m(\Omega)\Vert^2_{(b_\alpha)}\dint z}{\int\limits_{B_r} \Vert h_z-Ph_z |L_2(\Omega)\Vert^2 \dint z}
\] 
where we put $B_r:=\{z\in\rd: a(z)<r^{2m}\}$ for an arbitrary $r>0.$ We have
\begin{align*}
\Vert h_z-Ph_z |W^m(\Omega)\Vert^2_{(b_\alpha)} 
&=
\Vert h_z|W^m(\Omega)\Vert^2_{(b_\alpha)} - 2 \Re\text{e} \langle h_z - Ph_z,Ph_z\rangle_{W^m(\Omega),(b_\alpha)} - \Vert Ph_z|W^m(\Omega)\Vert^2_{(b_\alpha)} \\
&=
\Vert h_z|W^m(\Omega)\Vert^2_{(b_\alpha)} - \Vert Ph_z|W^m(\Omega)\Vert^2_{(b_\alpha)}\\
&=
\vol\Omega\,(1+a(z)) - \sum_{j=1}^k (1+ \mu_j)|\langle h_z,\phi_j\rangle_{L_2(\Omega)}|^2
\end{align*}
where the second term on the right-hand side in the first line vanishes since
\begin{align*}
\langle h_z - Ph_z,Ph_z\rangle_{W^m(\Omega),(a_\alpha)}
= 
\sum_{j=1}^\infty(1+\mu_j)
\underbrace{\langle h_z - Ph_z,\phi_j\rangle_{L_2(\Omega)}}_{=0 \text{ if } 1\leq j\leq k}\underbrace{\overline{\langle Ph_z,\phi_j\rangle_{L_2(\Omega)}}}_{=0\text{ if } j>k} =0.
\end{align*}
It follows that
\begin{equation}\label{xeq1}
\mu_{k+1}
\leq 
\frac{\vol\Omega\int\limits_{B_r} a(z)\dint z - \sum\limits_{j=1}^k\mu_j\int\limits_{B_r}|\langle h_z,\phi_j\rangle_{L_2(\Omega)}|^2\dint z}{\vol\Omega\int\limits_{B_r} \dint z - \sum\limits_{j=1}^k\int\limits_{B_r}|\langle h_z,\phi_j\rangle_{L_2(\Omega)}|^2\dint z}.
\end{equation}
Note that $\mu_j\geq 0$ and $\int_{B_r}|\langle h_z,\phi_j\rangle_{L_2(\Omega)}|^2\dint z \leq (2\pi)^d$ for any $j\in\nat.$ 
We choose $r>0$ such that the denominator in \eqref{xeq1} is positive by requiring 
\begin{align*}
(2\pi)^d\, k < \vol\Omega\int_{B_r}\dint z = \vol\Omega\cdot \vol A^d_m\cdot r^d.
\end{align*}
In particular we assume 
\begin{align*}
r^d:=  \frac{(2\pi)^d}{\vol\Omega}\, \frac{1}{\vol A^d_m}\, k\cdot \gamma
\end{align*}
for some $\gamma > 1.$ Then it follows from \eqref{xeq1} that
\begin{align}
\notag
\mu_{k+1}\cdot\vol\Omega\int_{B_r}\dint z 
&\leq
\vol\Omega\int_{B_r}a(z)\dint z + \sum_{j=1}^k(\mu_{k+1}-\mu_j)\int_{B_r}|\langle h_z,\phi_j\rangle_{L_2(\Omega)}|^2\dint z\\ \label{mu_k+1}
&\leq
\vol\Omega\int_{B_r}a(z)\dint z + (2\pi)^d k\cdot\mu_{k+1}.
\end{align}
Using Lemma \ref{Lemma_V} we obtain
\begin{align*}
\mu_{k+1}
\leq
\left(\frac{\vol \Omega}{(2\pi)^d}\right)^{-\frac{2m}{d}}\frac{d}{2m+d}  \big(\vol A^d_m\big)^{-\frac{2m}{d}}\, k^{\frac{2m}{d}}\cdot \frac{\gamma^{\frac{2m+d}{d}}}{\gamma-1}.
\end{align*}
Finally the optimal choice of $\gamma=\frac{2m+d}{2m}>1$ leads to \eqref{Kroe_lowBd} since $a_{k+1}=(1+\mu_{k+1})^{-\frac{1}{2}}$.
\end{proof}

\begin{rem}
In \eqref{mu_k+1} one can also chose $r^d:=  \frac{(2\pi)^d}{\vol\Omega}\, \frac{1}{\vol A^d_m}\,(k+1)$ to get sharp upper estimates even for the sum of the first $k$ eigenvalues, namely it comes out
\begin{equation}
\sum_{j=1}^k\mu_j\leq \frac{d}{2m+d}\Big( \frac{\vol\Omega}{(2\pi)^d}\vol A^d_m\Big)^{-\frac{2m}{d}} k^{\frac{2m+d}{d}}.
\end{equation}
\end{rem}

\section{Approximation in $W^{m,*}([0,L]^d)$}

Let $Q=[0,L]^d$ be the $d$-dimensional cube with side length $L>0$. Let us consider the Sobolev norm
\begin{equation}\label{def_Sobnorm*}
\Vert f| W^m(Q)\Vert ^* := \Big(\Vert f|L_2(Q)\Vert^2\,+\,\sum\limits_{j=1}^d \Vert \partial^m_jf|L_2(Q)\Vert^2\Big)^{1/2}
\end{equation}
where only the highest partial derivatives in each direction are considered. In order to analyse the behaviour of the approximation numbers $a_k(W^{m,*}(Q)\hookrightarrow L_2(Q))$ we need to study the eigenvalues of the norm-inducing operator 
\[
\mathcal{B}^* f:= (-1)^m\sum_{j=1}^d \partial^{2m}_jf,\qquad f\in C^{m,\nu}(Q),
\]
cf. Proposition \ref{prop_sobolevW^m}. The special characteristic here lays in the fact that the eigenfunctions of this norm-inducing operator are tensor products of the correspondig one-dimensional ones. This fact admits the use of techniques from \cite{KSU14} where the periodic case is treated. Therefor we have a closer look on the one-dimensional setting first. Consider the eigenvalue problem of the univariate polyharmonic operator on the interval $[0,L]$ equipped with homogeneous Dirichlet boundary conditions
\begin{equation}\label{1dim_Dbc}
\begin{split}
(-1)^m \psi^{(2m)}(t)=\lambda\, \psi(t),\qquad & t\in[0,L]\\
\psi^{(j)}(0)=\psi^{(j)}(L)=0,\qquad		& j=0,...,m-1
\end{split}
\end{equation}
where $\psi\in C^\infty(0,L), \lambda>0$ or respectively with homogeneous Neumannn boundary conditions
\begin{equation}\label{1dim_Nbc}
\begin{split}
(-1)^m \phi^{(2m)}(t)=\mu\, \phi(t),\qquad & t\in[0,L]\\
\phi^{(j)}(0)=\phi^{(j)}(L)=0,\qquad		& j=m,...,2m-1
\end{split}
\end{equation}
where $\phi\in C^\infty(0,L), \mu\geq 0$. First we collect some observations and properties. We have
\begin{equation}
\lambda=0 \textit{ is not an eigenvalue of \eqref{1dim_Dbc}}
\end{equation} 
and
\begin{equation}
\mu=0 \textit{ is an $m$-fold eigenvalue of \eqref{1dim_Nbc}}
\end{equation} 
which is clear since polynomials of degree less than $m$ are corresponding eigenfunctions (and no others). Furthermore
\begin{equation}\label{A2}
\textit{ all positive eigenvalues of \eqref{1dim_Nbc} coincide with the eigenvalues of \eqref{1dim_Dbc}.}
\end{equation}
Indeed, let $\phi\in C^\infty(0,L)$ solve \eqref{1dim_Nbc} with some $\mu>0$. Then $\psi(t):=\phi^{(m)}(t)$ is an eigenfunction of \eqref{1dim_Dbc} for the same eigenvalue $\lambda=\mu$. One can easily verify that an eigenvalue of \eqref{1dim_Dbc} has the same multiplicity in \eqref{1dim_Nbc} and vice versa. Hence if 
\[
0<\lambda_1\leq\lambda_2\leq... \leq \lambda_n \overset{n\to\infty}{\longrightarrow}\infty
\]
are the eigenvalues of \eqref{1dim_Dbc} counted according to their multiplicity and respectively
\[
0=\mu_1=...=\mu_m<\mu_{m+1}\leq\mu_{m+2}\leq... \leq \mu_n \overset{n\to\infty}{\longrightarrow}\infty
\]
are the eigenvalues of \eqref{1dim_Nbc} then we have $\mu_n = \lambda_{n-m}$ for $n>m$. The asymptotic behaviour of the one-dimensional polyharmonic eigenvalues is given by
\begin{equation}\label{A3}
\qquad\qquad\big(\tfrac{\pi}{L} n\big)^{2m}\leq \lambda_n=\mu_{n+m} \leq \big(\tfrac{\pi}{L}(n+m-1)\big)^{2m}\qquad \forall\, n\in\nat.
\end{equation}
To verify the last assertion one defines functions $f_k(t):=\sin^{m-1}(\pi t)\sin(\pi k t), k\in\nat.$ For simplicity assume $[0,L]=[0,1]$. The functions $f_k$ satisfy Dirichlet boundary conditions and belong to the finite-dimensional subspace $\mathcal{M}_{k+m-1}$ defined by
\[
\mathcal{M}_N:=\spa\{1,\sin(j\pi t),\cos(j\pi t): j=1,...,N\},\quad N\in\nat. 
\]
Note that for any $g\in \mathcal{M}_N$, say $g(t)=\alpha_0 +\sum\limits_{j=1}^N\left(\alpha_j\sin(j\pi t) +\beta_j\cos(j\pi t) \right)$, we get
\[
\Vert g^{(m)}| L_2(0,1)\Vert^2 
=
\sum_{j=1}^N(j\pi)^{2m}\left(\alpha_j^2 +\beta_j^2\right) 
\leq
(N\pi)^{2m}\Vert g| L_2(0,1)\Vert^2.
\]
Therefore by Rayleigh-Ritz variational formula \eqref{varchar_D} we obtain
\[
\lambda_n 
\leq
\sup_{g\in\spa\{f_1,...,f_n\}} 
\frac{\Vert g^{(m)}\,| L_2(0,1)\Vert^2}{\Vert g\,|L_2(0,1)\Vert^2}
\leq \big(\pi(n+m-1)\big)^{2m}.
\] 
On the other hand consider the $m$-times iterated Dirichlet-Laplace eigenvalue problem
\begin{equation}\label{1dim_itDbc}
\begin{split}
(-1)^m v^{(2m)}(t)=\nu\, v(t),\qquad & t\in[0,1]\\
v^{(2j)}(0)=v^{(2j)}(1)=0,\qquad		& j=0,...,m-1
\end{split}
\end{equation}
where the eigenvalues are simply given by the $m$th power of the Dirchlet-Laplace eigenvalues, i.e.
\[
\nu_n=(\pi n)^{2m},\quad n\in\nat.
\]
Comparing the corresponding quadratic form domains of \eqref{1dim_Dbc} and \eqref{1dim_itDbc}, we see that $\nu_n\leq\lambda_n.$ This completes the proof of \eqref{A3}.

\begin{prop}\label{prop_norm*}
Suppose that $a_k(\id^*)$ denotes the $k^{\textup{th}}$ approximation number of the compact embedding $\id^*:W^{m,*}(Q)\hookrightarrow L_2(Q)$
where the Sobolev norm $\Vert\cdot| W^m(Q)\Vert^*$ is given by \eqref{def_Sobnorm*}. Then $(a_k(\id^*))_{k\in\nat}$ is the non-increasing rearrangement of 
\begin{equation}\label{anbar}
a_{\bar{n}}(\id^*) 
= \Big(1+\sum_{j=1}^d \mu_{n_j} \Big)^{-\frac{1}{2}},\quad \bar{n}=(n_1,...,n_d)\in\nat^d,
\end{equation}
where $(\mu_n)_{n\in\nat}$ are the univariate eigenvalues of the polyharmonic operator with Neumann boundary conditions \eqref{1dim_Nbc}.
\end{prop}

\begin{proof}
Let $\{\phi_n:n\in\nat\}$ be the complete orthonormal set of eigenfunctions in $L_2(0,L)$ corresponding to the eigenvalues $(\mu_n)_{n\in\nat}$. For $\bar{n}=(n_1,...,n_d)\in\nat^d$ and $x=(x_1,...,x_d)\in [0,L]^d$ we define
\[
f_{\bar{n}}(x):= \prod_{j=1}^d \phi_{n_j}(x_j).
\]
Then $\{f_{\bar{n}}: \bar{n}\in\nat^d\}$ is an orthonormal basis of $L_2([0,L]^d)$. Furthermore $f_{\bar{n}}$ are (all) eigenfunctions of the operator $\mathcal{B}^*=(-1)^m\sum\limits_{i=1}^d \partial^{2m}_i$ since
\begin{align*}
(\mathcal{B}^* f_{\bar{n}})(x) 
= \sum_{i=1}^d \bigg(\prod_{\substack{j=1\\j\neq i}}^d \phi_{n_j}(x_j) \bigg) (-1)^m\phi_{n_i}^{(2m)}(x_i)
=
\Big(\sum_{i=1}^d \mu_{n_i}\Big)\,f_{\bar{n}}(x).
\end{align*}
Finally \eqref{ak} shows \eqref{anbar}.
\end{proof}

\begin{theo}\label{theo_norm*}
Let $Q=[0,L]^d$. The following hold
\begin{itemize}
\item[(i)] 
\emph{Curse of dimension: }
\begin{equation}
a_k(\id^*: W^{m,*}(Q)\hookrightarrow L_2(Q))=1,\qquad 1\leq k\leq m^d
\end{equation}
\item[(ii)] 
\emph{Asymptotic constants: }
\begin{equation}
\lim_{k\to\infty}\frac{a_k(\id^*: W^{m,*}(Q)\hookrightarrow L_2(Q))}{k^{-\frac{m}{d}}}=\Big(\frac{L}{2\pi}\Big)^m\, \big(\vol B^d_{2m} \big)^{\frac{m}{d}}
\end{equation}
\item[(iii)] 
\emph{Explicit estimates for large $k$: }
\begin{align}\label{explicit*_up}
\forall\, k\geq (2m)^d (d^{\frac{1}{2m}}+1)^d \vol B^d_{2m}:\quad 
&a_k(\id^*)\leq  \Big(\frac{L}{\pi}\Big)^m \big(\vol B^d_{2m} \big)^{\frac{m}{d}}\cdot k^{-\frac{m}{d}}\\\label{explicit*_low}
\forall\, k\geq \Big(\tfrac{1}{2}+\tfrac{m+d^{\frac{1}{2m}}}{2} + \tfrac{L}{2\pi} \Big)^d \vol B^d_{2m}:\quad 
&a_k(\id^*)\geq \Big(\frac{L}{4\pi}\Big)^m \big(\vol B^d_{2m} \big)^{\frac{m}{d}}\cdot k^{-\frac{m}{d}}
\end{align}
\end{itemize}
\end{theo}

\begin{lemma}\label{lemma_A(r)}
Let $1\leq p<\infty.$ Let $A(r):=\#\{\bar{k}\in\nat^d: \Vert \bar{k}|\ell_p^d\Vert \leq r\}, r> d^{\frac{1}{p}}$. Then it holds
\[
2^{-d}(r-d^{\frac{1}{p}})^d \vol B^d_p\leq A(r)\leq 2^{-d}(r+d^{\frac{1}{p}})^d \vol B^d_p.
\]
\end{lemma}
\begin{proof}
For $\bar{k}\in\nat^d$ we put $Q_{\bar{k}}:= \bar{k}+[-1,0]^d.$ The assertion follows from the set inclusions
\[
\left\{y\in\rd_+ : \Vert y\Vert_p\leq r-d^{\frac{1}{p}}\right\}
\subseteq
\bigcup_{\substack{\bar{k}\in\nat^d\\ \Vert\bar{k}\Vert_p\leq r}} Q_{\bar{k}}
\subseteq
\left\{y\in\rd_+ : \Vert y\Vert_p\leq r+d^{\frac{1}{p}}\right\}.
\]
\end{proof}

\begin{proof}[Proof of Theorem \ref{theo_norm*}]
(i) is clear, since the value $1+ \sum\limits_{i=1}^d\mu_{n_i}$ with $
(n_1,...,n_d)\in\nat^d$ is $m^d$-times assumed to be $1$ because $\mu_1=...=\mu_m=0$. We exploit Proposition \ref{prop_norm*} and adapt techniques from \cite{KSU14}. We define the cardinality
\[
C(l,d):= \#\big\{\bar{n}\in\nat^d: a_{\bar{n}}(\id^*)\geq (1+\mu_l)^{-\frac{1}{2}}\big\},
\quad l,d\in\nat,
\]
and conclude that
\begin{equation}\label{eq_a_C}
a_{C(l,d)}(\id^*)=(1+\mu_l)^{-\frac{1}{2}}.
\end{equation}
With \eqref{A3} and Lemma \ref{lemma_A(r)} we observe
\begin{align*}
C(l,d)
&=
\#\big\{\bar{n}\in\nat^d: \sum_{i=1}^d\mu_{n_i}\leq \mu_l\big\}
\leq
\#\Big\{\bar{n}\in\nat^d: \Big(\sum_{i=1}^d\big(n_i-m \big)^{2m}_{+}\Big)^{\frac{1}{2m}}\leq l-1 \Big\} \\
&\leq
\#\big\{\bar{n}\in\nat^d: \Vert\bar{n}\Vert_{2m} \leq l-1 +d^{\frac{1}{2m}}m\big\}
\leq
2^{-d}\big(l-1 + d^{\frac{1}{2m}}(m+1)\big)^d\vol B^d_{2m}
\end{align*}
where the second inequality follows from the triangle inequality in $\ell_{2m}^d$ and the simple fact that
$n_i=(n_i-m)_{+} + x_i$ where $x_i:=\begin{cases} m,&\text{ if }n_i>m\\ n_i,&\text{ if } n_i\leq m\end{cases}$. Furthermore for any $l> d^{\frac{1}{2m}}+m$
\begin{align*}
C(l,d)
&=
\#\big\{\bar{n}\in\nat^d: \sum_{i=1}^d\mu_{n_i}\leq \mu_l\big\}
\geq
\#\Big\{\bar{n}\in\nat^d: \Big(\sum_{i=1}^d n_i^{2m}\Big)^{\frac{1}{2m}}\leq l-m \Big\} \\
&\geq
2^{-d}\big(l-m - d^{\frac{1}{2m}}\big)^d\vol B^d_{2m} 
\end{align*}
Hence relation \eqref{eq_a_C} and the monotonicity of approximation numbers yield
\begin{align*}
&\forall\ k\geq A(l,d):= 2^{-d}\big(l-1 + d^{\frac{1}{2m}}(m+1)\big)^d\vol B^d_{2m}:\
&a_k(\id^*)\leq (1+\mu_l)^{-\frac{1}{2}} \\
&\forall\ k\leq B(l,d):= 2^{-d}\big(l-m - d^{\frac{1}{2m}}\big)^d\vol B^d_{2m}:\ 
&a_k(\id^*)\geq (1+\mu_l)^{-\frac{1}{2}}.
\end{align*}
Based on this lower and upper estimates we  prove (ii) and (iii). Let $l\in\nat$ with $A(l,d)\leq k< A(l+1,d).$ Then
\begin{equation}\label{ineq_1}
k^{\frac{m}{d}}a_k(\id^*)
\leq 
\big[A(l+1,d)\big]^{\frac{m}{d}} (1+\mu_l)^{-\frac{1}{2}}
\leq
\frac{2^{-m} \big(l+d^{\frac{1}{2m}}(m+1)\big)^m}{\big(1+(\frac{\pi}{L}(l-m)_{+})^{2m}\big)^{\frac{1}{2}}} \big(\vol B^d_{2m}\big)^{\frac{m}{d}}
\end{equation}
where the right hand side of the inequality tends to $\big(\frac{L}{2\pi} \big)^m \big(\vol B^d_{2m}\big)^{\frac{m}{d}}$ as $l\to\infty.$ A similar lower inequality finally shows (ii). Now we fix $l_0:= \lceil d^{\frac{1}{2m}}(m+1)+2m\rceil$ and $k\geq A(l_0,d).$ Then for all $l\in\nat, l\geq l_0,$ satisfying $A(l,d)\leq k < A(l+1,d)$, we go ahead  with \eqref{ineq_1} and conclude
\[
k^{\frac{m}{d}}a_k(\id^*)
\leq 
\Big( \frac{L}{2\pi}\Big)^{m} \Big( \frac{l+d^{\frac{1}{2m}}(m+1)}{l-m}\Big)^m \big(\vol B^d_{2m}\big)^{\frac{m}{d}}
\leq
\Big( \frac{L}{2\pi}\Big)^{m} 2^m \big(\vol B^d_{2m}\big)^{\frac{m}{d}}
\]
for large enough $k\geq A(l_0,d).$ The particular choice of $l_0$  implies 
\[
A(l_0,d) 
\leq 
\big(d^{\frac{1}{2m}}(m+1)+m\big)^d\vol B^d_{2m}
\leq
(2m)^d (d^{\frac{1}{2m}}+1)^d \vol B^d_{2m}.
\]
This proves the explicit upper estimate in (iii). Let $l_1:=\lceil 1+2(m+d^{\frac{1}{2m}})+\frac{L}{\pi}\rceil$ and $k\geq B(l_1-1,d).$ Then it holds for any $l\in\nat, l\geq l_1,$ with $B(l-1,d)\leq k<B(l,d)$
\begin{align*}
k^{\frac{m}{d}}a_k(\id^*)
&\geq
\big[B(l-1,d)\big]^{\frac{m}{d}} (1+\mu_l)^{-\frac{1}{2}}
\geq
\Big(\frac{L}{2\pi}\Big)^{m}\Big(\frac{l-1-m-d^{\frac{1}{2m}}}{l-1+\frac{L}{\pi}} \Big)^m
\big(\vol B^d_{2m}\big)^{\frac{m}{d}}\\
&\geq
\Big(\frac{L}{2\pi}\Big)^{m} 2^{-m} \big(\vol B^d_{2m}\big)^{\frac{m}{d}}
\end{align*}
where we used \eqref{A3} and $(1+x^{2m})^{\frac{1}{2}}\leq (1+x)^m, x\geq 0.$ Furthermore,
\[
B(l_1-1,d)
\leq 
2^{-d}\big(1+m+d^{\frac{1}{2m}}+\tfrac{L}{\pi} \big)^d \vol B^d_{2m}.
\]
\end{proof}

\begin{rem}
In (iii) of Theorem \ref{theo_norm*} one can improve the factor $2^m$ (or resp. $2^{-m}$) by $(1+\varepsilon)^m$ (or resp. $(1-\varepsilon)^m$) for an arbitrary small $0<\varepsilon\leq 1$ by enlarging the index $k.$ In fact we have
\begin{equation}
\forall k\geq k_0(\varepsilon):
\ a_k(\id^*)\, k^{\frac{m}{d}}
\leq  
\Big(\frac{L}{2\pi}\Big)^m (1+\varepsilon)^m\big(\vol B^d_{2m} \big)^{\frac{m}{d}}
\end{equation}
where 
\[
k_0(\varepsilon):=\Big(\frac{1+\varepsilon}{\varepsilon}\Big)^d m^d (d^{\frac{1}{2m}}+1)^d \vol B^d_{2m}\overset{\varepsilon\to 0}{\longrightarrow}\infty
\] 
and respectively
\begin{equation}
\forall k\geq k_1(\varepsilon):
\ a_k(\id^*)\, k^{\frac{m}{d}}
\geq  
\Big(\frac{L}{2\pi}\Big)^m (1-\varepsilon)^m\big(\vol B^d_{2m} \big)^{\frac{m}{d}}
\end{equation}
where 
\[
k_1(\varepsilon)
:=
\Big(\tfrac{1}{2}+\big(\tfrac{1-\varepsilon}{\varepsilon}\big)\big(\tfrac{m+d^{\frac{1}{2m}}}{2} + \tfrac{L}{2\pi} \big)\Big)^d \vol B^d_{2m}\overset{\varepsilon\to 0}{\longrightarrow}\infty.
\] 
Note that the index bounds $k_0(1)$ and $k_1(1)$ are worse if we compared them to the periodic results \cite{KSU14}[Theorem 4.12]. There the authors showed the upper estimate \eqref{explicit*_up} with $\varepsilon=1$ for all $k\geq (d^{\frac{1}{2m}}+2)^d\vol B^d_{2m}$ and the lower estimate \eqref{explicit*_low} with $\varepsilon=\tfrac{1}{2}$ for all $k\geq \big(3+ \frac{d^{\frac{1}{2m}}}{2}\big)^d\vol B^d_{2m}.$ This deviation is a consequence of the fact that we have to use \eqref{A3} while in the periodic case the eigenvalues are known explicitly.

\end{rem}

\begin{rem}
As in \cite[Lemma 4.10]{KSU14} one can estimate the volume of the unit ball $B^d_{2m}$ by
\[
2^d \big(e(d+2m)\big)^{-\frac{d}{2m}}
\leq
\vol B^d_{2m}
\leq
2^d \Big(\frac{2em}{d} \Big)^{\frac{d}{2m}}.
\]
Note that
\[
\lim_{d\to\infty} \sqrt{d}\cdot\big(\vol B^d_{2m} \big)^{\frac{m}{d}}
=2^{m}\cdot \sqrt{2em}\cdot \Gamma(1+\tfrac{1}{2m})^{m}.
\]
In the particular case of $[0,L]=[0,2\pi]$ this gives
\begin{align}
\forall\, k\geq (6m+3)^d e^{\frac{d}{2m}}:\quad 
&a_k(\id^*)\leq  \frac{4^m}{\sqrt{d}} \sqrt{2em}\cdot k^{-\frac{m}{d}}\\
\forall\, k\geq (\tfrac{3}{2}m+6)^d  e^{\frac{d}{2m}}:\quad 
&a_k(\id^*)\geq \frac{1}{\sqrt{e(d+2m)}}\cdot k^{-\frac{m}{d}}.
\end{align}
\end{rem}

\paragraph*{Acknowledgements}
This work was supported by Deutsche Forschungsgemeinschaft (DFG) [project number 365818999].

\end{document}